\documentclass[12pt, twoside]{article}
\usepackage{amsmath,amsthm,amssymb, graphicx}
\usepackage{times}
\usepackage{enumerate}

\def\titlerunning#1{\gdef\titrun{#1}}
\makeatletter
\def\author#1{\gdef\autrun{\def\and{\unskip, }#1}\gdef\@author{#1}}
\def\address#1{{\def\and{\\\hspace*{18pt}}\renewcommand{\thefootnote}{}%
\footnote {#1}}%
\markboth{\autrun}{\titrun}}
\makeatother
\def\email#1{e-mail: #1}
\def\subjclass#1{{\renewcommand{\thefootnote}{}%
\footnote{\emph{Mathematics Subject Classification (2010):} #1}}}
\def\keywords#1{\par\medskip
\noindent\textbf{Keywords.} #1}

\newtheorem{thm}{Theorem}[section]
\newtheorem{cor}[thm]{Corollary}
\newtheorem{lem}[thm]{Lemma}

\newtheorem{prop}[thm]{Proposition}

\theoremstyle{definition}

\newtheorem{rem}[thm]{Remark}
\newtheorem{exa}[thm]{Example}

\numberwithin{equation}{section}

\begin{document}


\baselineskip=17pt


\titlerunning{Local dimensions for the random $\beta$-transformation}

\title{Local dimensions for the random $\beta$-transformation}
\author{Karma Dajani
\and
Charlene Kalle}

\date{}

\maketitle

\address{K. Dajani: Department of Mathematics, Utrecht University, Postbus 80.000, 3508 TA Utrecht, the Netherlands; \email{k.dajani1@uu.nl}
\and
C. Kalle: Mathematical Institute, Leiden University, Postbus 9512, 2300 RA Leiden, the Netherlands; \email{kallecccj@math.leidenuniv.nl}}



\begin{abstract}
The random $\beta$-transformation $K$ is isomorphic to a full shift. This relation gives an invariant measure for $K$ that yields the Bernoulli convolution by projection. We study the local dimension of the invariant measure for $K$ for special values of $\beta$ and use the projection to obtain results on the local dimension of the Bernoulli convolution.
\end{abstract}

\newcommand{\fb}{\frac{\lfloor \beta \rfloor}{\beta-1}}
\newcommand{\bint}{\big[0,\frac{\lfloor \beta \rfloor}{\beta-1}\big]}

\section{Introduction}
The Bernoulli convolution has been around for over seventy years and has surfaced in several different areas of mathematics. This probability measure depends on a parameter $\beta >1$ and is defined on the interval $\big[ 0, \fb \big]$, where $\lfloor \beta \rfloor$ is the largest integer not exceeding $\beta$. The {\bf symmetric Bernoulli convolution} is the distribution of $\sum_{k=1}^{\infty} \frac{b_k}{\beta_k}$ where the coefficients $b_k$ take values in the set $\{ 0 ,1, \ldots, \lfloor \beta \rfloor \}$, each with probability $\frac{1}{\lfloor \beta \rfloor +1}$. If instead the values $0, 1, \ldots, \lfloor \beta \rfloor$ are not taken with equal probabilities, then the Bernoulli convolution is called {\bf asymmetric} or {\bf biased}. See \cite{PSS00} for an overview of results regarding the Bernoulli convolution up to the year 2000. Recently attention has shifted to the multifractal structure of the Bernoulli convolution. Jordan, Shmerkin and Solomyak study the multifractal spectrum for typical $\beta$ in \cite{JSS11}, Feng considers Salem numbers $\beta$ in \cite{Fen12} and Feng and Sidorov look at the Lebesgue generic local dimension of the Bernoulli convolution in \cite{FS11}. In this paper we are interested in the local dimension function of the Bernoulli convolution and to study the local dimension we use a new approach.

If a point $x \in \big[ 0, \fb \big]$ can be written as $x = \sum_{k=1}^{\infty} \frac{b_k}{\beta^k}$ with $b_k \in \{0,1, \ldots, \lfloor \beta \rfloor\}$ for all $k \ge 1$, then this expression is called a $\beta$-expansion of the point $x$. In \cite{S03} (see also \cite{DV05}) it is shown that Lebesgue almost every $x$ has uncountably many $\beta$-expansions. In \cite{DV05} a random transformation $K$ was introduced that generates for each $x$ all these possible expansions by iteration. The map $K$ can be identified with a full shift which allows one to define an invariant measure $\nu_{\beta}$ for $K$ of maximal entropy by pulling back the uniform Bernoulli measure. One obtains the Bernoulli convolution from $\nu_{\beta}$ by projection. In this paper we study the local dimension of the measure $\nu_{\beta}$. By projection, some of these results can be translated to the Bernoulli convolution. For now, our methods work only for a special set of $\beta$'s called the generalised multinacci numbers. We have good hopes that in the future we can extend these methods to a more general class of $\beta$'s.

The paper is organized as follows. In the first section we will give the necessary definitions. Next we study the local dimension of $\nu_{\beta}$. We give a formula for the lower and upper bound of the local dimension that holds everywhere using a suitable Markov shift. Moreover, we show that the local dimension exists and is constant a.e.~and we give this constant. We also show that on the set corresponding to points with a unique $\beta$-expansion, the local dimension of $\nu_{\beta}$ takes a different value. Next we translate these results to a lower and upper bound for the local dimension of the symmetric Bernoulli convolution that holds everywhere. We then use a result from \cite{FS11} to obtain an a.e.~value for the Bernoulli convolution in case $\beta$ is a Pisot number. Finally we give the local dimension for points with a unique expansion. In the last section we consider one specific example of an asymmetric Bernoulli convolution, namely when $\beta$ is the golden ratio. We give an a.e.~lower and upper bound for the local dimension of both the invariant measure for $K$ and the asymmetric Bernoulli convolution. This last section is just a starting point for more research in this direction.

\section{Preliminaries}

The set of $\beta$'s we consider in this paper, the generalised multinacci numbers, are defined as follows. On the interval $\big[0,\fb \big]$ the {\bf greedy} $\beta$-transformation $T_{\beta}$ is given by
\[ T_{\beta} x = \left\{
\begin{array}{ll}
\beta x \, (\text{mod }1), & \text{if } x \in \big[0,1),\\
\beta x - \lfloor \beta \rfloor, & \text{if } x \in \big[1, \fb \big].
\end{array}
\right.\]
The {\bf greedy digit sequence} of a number $x \in \big[ 0, \fb \big]$ is defined by setting 
\[ a_1=a_1(x) = \left\{
\begin{array}{ll}
k, & \text{if } x \in \big[\frac{k}{\beta}, \frac{k+1}{\beta} \big), \, k \in \{0, \ldots, \lfloor \beta \rfloor -1 \},\\
\lfloor \beta \rfloor, & \text{if } x \in \big[ \frac{\lfloor \beta \rfloor}{\beta}, \fb \big],
\end{array}
\right.\]
and for $n \ge 1$, $a_n=a_n(x) = a_1(T_{\beta}^{n-1} x)$. Then $T_{\beta}x = \beta x - a_1(x)$ and one easily checks that $x = \sum_{n=1}^{\infty} \frac{a_n}{\beta^n}$. This $\beta$-expansion of $x$ is called its {\bf greedy $\beta$-expansion}. A number $\beta >1$ is called a {\bf generalised multinacci number} if the greedy $\beta$-expansion of the number 1 satisfies
\begin{equation}\label{q:genmn}
1 = \frac{a_1}{\beta} + \frac{a_2}{\beta^2} + \cdots + \frac{a_n}{\beta^n},
\end{equation}
with $a_j \ge 1$ for all $1 \le j \le n$ and $n \ge 2$. (Note that $a_1 = \lfloor \beta \rfloor$.) We call $n$ the {\bf degree} of $\beta$.

\begin{rem}{\rm
Between 1 and 2  the numbers that satisfy this definition are called the multinacci numbers. The {\bf $n$-th multinacci number} $\beta_n$ satisfies
\[ \beta^n_n = \beta^{n-1}_n + \beta^{n-2}_n + \cdots + \beta_n + 1,\]
which implies that $a_j=1$ for all $1 \le j \le n$ in (\ref{q:genmn}). The second multinacci number is better known as the golden ratio.}
\end{rem}

\vskip .2cm
For the Markov shift we will construct later on, we need a suitable partition of the interval $\bint$. Consider the maps $T_k x = \beta x -k$, $k =0, \ldots, \lfloor \beta \rfloor$. For each $x \in \bint$, either there is exactly one $k \in \{0, \ldots, \lfloor \beta \rfloor \}$ such that $T_k x \in \bint$, or there is a $k$ such that both $T_k x$ and $T_{k+1} x$ are in $\bint$. In this way the maps $T_k$ partition the interval $\bint$ into the following regions:
\[ \begin{array}{ll}
\displaystyle E_0 = \Big[0, \frac{1}{\beta} \Big), & \displaystyle E_{\lfloor \beta \rfloor} = \Big(  \frac{\lfloor \beta \rfloor}{\beta(\beta -1)} +\frac{\lfloor \beta \rfloor-1}{\beta}, \fb \Big],\\
\\
\displaystyle E_k = \Big( \frac{\lfloor \beta \rfloor}{\beta(\beta -1)} +\frac{k-1}{\beta}, \frac{k+1}{\beta}\Big), & k \in \{ 0, 1, \ldots, \lfloor \beta \rfloor \},\\
\\
\displaystyle S_k = \Big[ \frac{k}{\beta}, \frac{\lfloor \beta \rfloor}{\beta(\beta-1)} + \frac{k-1}{\beta} \Big], & k \in \{1, \ldots, \lfloor \beta \rfloor \}.
\end{array}\]
See Figure~\ref{f:2.5} for a picture of the maps $T_k$ and the regions $E_k$ and $S_k$ in case $2 < \beta < 3$.
\begin{figure}[ht]
\centering
\includegraphics{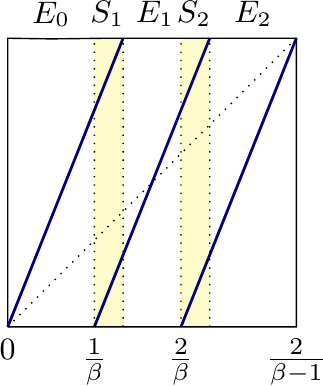}
\caption{The maps $T_0 x = \beta x$, $T_1 x = \beta x-1$ and $T_2 x = \beta x -2$ and the intervals $E_0$, $S_1$, $E_1$, $S_2$ and $E_2$ for some $2< \beta <3$.}
\label{f:2.5}
\end{figure}

\vskip .2cm
Write $\Omega = \{0,1 \}^{\mathbb N}$. The {\bf random $\beta$-transformation} is the map $K$ from the space $\Omega \times \bint$ to itself defined as follows.
\[ K(\omega, x) = \left\{
\begin{array}{ll}
(\omega, T_k x), & \text{if } x \in E_k, \, k \in \{0, \ldots, \lfloor \beta \rfloor \},\\
(\sigma \omega, T_{k-1+\omega_1} x), & \text{if }x \in S_k, \, k \in \{ 1, \ldots, \lfloor \beta \rfloor\},
\end{array}
\right.\]
where $\sigma$ denotes the left shift on sequences, i.e., $\sigma (\omega_n)_{n \ge 1} = (\omega_{n+1})_{n \ge 1}$. The projection onto the second coordinate is denoted by $\pi$. Let $\lceil \beta \rceil$ denote the smallest integer not less than $\beta$. The map $K$ is isomorphic to the full shift on $\lceil \beta \rceil$ symbols. The isomorphism $\phi: \Omega \times \bint \to \{0,1, \ldots, \lfloor \beta \rfloor\}^{\mathbb N}$ uses the digit sequences produced by $K$. Let
\[b_1(\omega,x) = \left\{
\begin{array}{ll}
k, & \text{if } x \in E_k, \, k \in \{0, 1, \ldots, \lfloor \beta \rfloor \},\\
& \  \text{or if } x \in S_k \text{ and } \omega_1=1, \, k \in \{1, \ldots, \lfloor \beta \rfloor \},\\
k-1, & \text{if } x \in S_k \text{ and } \omega_1=0, \, k \in \{ 1, \ldots, \lfloor \beta \rfloor \}
\end{array}
\right.\]
and for $n \ge 1$, set $b_n(\omega,x) = b_1\big(K^{n-1}(\omega, x)\big)$. Then
\[ \phi(\omega,x) = \big(b_n (\omega,x) \big)_{n \ge 1}.\]
This map is a bimeasurable bijection from the set $Z=\big\{ (\omega, x) \, : \, \pi \big(K^n (\omega, x)\big) \in S \, \, \text{i.o.} \big\}$ to its image. We have $\phi\circ K = \sigma \circ \phi$. Let $\mathcal F$ denote the $\sigma$-algebra generated by the cylinders and let $m$ denote the uniform Bernoulli measure on $(\{0,1, \ldots, \lfloor \beta \rfloor\}^{\mathbb N}, \mathcal F)$. Then $m$ is an invariant measure for $\sigma$ and $\nu_{\beta} = m \circ \phi$ is invariant for $K$ with $\nu_{\beta}(Z)=1$. The projection $\mu_{\beta} = \nu_{\beta} \circ \pi^{-1}$ is the Bernoulli convolution on $\bint$. For proofs of these facts and more information on the map $K$ and its properties, see \cite{DK03} and \cite{DV05}.

\vskip .2cm
We are interested in the local dimension of the measures $\nu_{\beta}$ and $\mu_{\beta}$. For any probability measure $\mu$ on a metric space $(X,\rho)$, define the {\bf local lower} and {\bf local upper dimension} functions by
\[ \underline{d}(\mu,x) = \liminf_{r \downarrow 0} \frac{\log \mu\big(B_{\rho}( x,r)\big)}{\log r} \quad \text{and} \quad \overline{d}(\mu,x) = \limsup_{r \downarrow 0} \frac{\log \mu\big(B_{\rho}(x,r)\big)}{\log r},\]
where $B_{\rho}(x,r)$ is the open ball around $x$ with radius $r$. If $\underline{d}(\mu,x) = \overline{d}(\mu,x)$, then the {\bf local dimension} of $\mu$ at the point $x \in X$ exists and is given by
\[ d(\mu,x)= \lim_{r \downarrow 0} \frac{\log \mu\big(B_{\rho}(x,r)\big)}{\log r}.\]
On the sets $\{0,1, \ldots, \lfloor \beta \rfloor\}^{\mathbb N}$ and $\Omega$ we define the metric $D$ by
\[ D \big( \omega, \omega' \big) = \beta^{-\min\{k\ge 0 \, :\, \omega_{k+1} \neq \omega_{k+1}'\}}.\]
We will define an appropriate metric on the set $\Omega \times \bint$ later.

\section{Local dimension for $\nu_{\beta}$}

We will study the local dimension of the invariant measure $\nu_{\beta}$ of the map $K$ for $\beta$'s that are generalised multinacci numbers. It is proven in \cite{DV05} that for these $\beta$'s the dynamics of $K$ can be modeled by a subshift of finite type. So, on the one hand there is the isomorphism of $K$ with the full shift on $\lceil \beta \rceil$ symbols and on the other hand there is an isomorphism to a subshift of finite type. It is this second isomorphism that allows us to code orbits of points $(\omega, x)$ under $K$ in an appropriate way for finding local dimensions. We give the essential information here.

\medskip
We begin with some notation. We denote the {\bf greedy map} by $T_{\beta}$ as before, and the {\bf lazy map} by $S_{\beta}$. More precisely, 
\[ T_{\beta} x = \left\{
\begin{array}{ll}
T_0 x, & \text{if } x \in E_0,\\
\\
T_k x, & \text{if }x \in  S_k\cup E_k,\\
& \quad 1 \le k \le \lfloor \beta \rfloor,
\end{array}
\right.
\quad \text{and} \quad S_{\beta}x = \left\{
\begin{array}{ll}
T_k x, & \text{if } x \in E_k\cup S_{k+1},\\
& \quad 0 \le k \le \lfloor \beta \rfloor-1,\\
\\
T_{\lfloor \beta \rfloor} x, & \text{if }x \in E_{\lfloor \beta \rfloor}.
\end{array}
\right.\]

\noindent
We will be interested in the $K$-orbit of the points $(\omega,1)$ and of their `symmetric counterparts' $(\omega, \frac{\lfloor \beta \rfloor}{\beta-1}-1)$. 
Proposition 2 (ii) in \cite {DV05} tells us that the following set $F$ is finite:
\begin{multline}\label{q:finite}
F=\Big\{\pi \big(K^n(\omega, 1)\big), \pi \Big(K^n\big(\omega, \frac{\lfloor \beta \rfloor}{\beta-1}-1\big)\Big) \, : \, n \ge 0,\, \omega \in \Omega\Big\}\\
\cup \Big\{\frac{k}{\beta}, \frac{\lfloor \beta \rfloor}{\beta(\beta -1)} + \frac{k}{\beta} : k \in \{0, \ldots, \lfloor \beta \rfloor \} \Big\}.
\end{multline}
These are the endpoints of the intervals $E_k$ and $S_k$ and their forward orbits under all the maps $T_k$. The finiteness of $F$ implies that the dynamics of $K$ can be identified with a topological Markov chain. To find the Markov partition, one starts by refining the partition given by the sets $E_k$ and $S_k$, using the points from the set $F$. Let $\mathcal C$ be the interval partition consisting of the open intervals between the points from this set. Note that when we say {\bf interval partition}, we mean a collection of pairwise disjoint open intervals such that their union covers the interval $\bint$ up to a set of $\lambda$-measure 0, where $\lambda$ is the one-dimensional Lebesgue measure. Write
\[ \mathcal C=\{C_1,C_2,\ldots , C_L\}.\]
Let $ S=\bigcup_{1 \le k \le \lfloor \beta \rfloor} S_k$. The property {\bf p3} from \cite{DV05} says that no points of $F$ lie in the interior of $S$, i.e., each $S_k$ corresponds to a set $C_j$ in the sense that for each $1 \le k \le \lfloor \beta \rfloor$ there is a $1 \le j \le L$ such that $\lambda(S_k \setminus C_j)=0$. Let $s \subset \{1, \ldots, L\}$ be the set containing those indices $j$. Consider the $L\times L$ adjacency matrix $A=(a_{i, j})$ with entries in $\{0,1\}$ defined by
\begin{equation}\label{matrix}
a_{i,j}\, =\, \left\{ \begin{array}{ll}
1, & {\mbox{ if }}\; i\not \in s \text{ and } \lambda(C_j \cap T_{\beta}(C_i))=\lambda(C_j),\\
0, & \text{ if }\; i\not \in s \text{ and } \lambda(C_i\cap T_{\beta}C_j)= 0,\\
1, & \text{ if }\; i \in s \text{ and } \lambda(C_j \cap T_{\beta} C_i)=\lambda(C_j) \text{ or } \lambda(C_j \cap S_{\beta}C_i)=\lambda(C_j),\\
0, & \text{ if }\; i \in s \text{ and } \lambda(C_i\cap T_{\beta}C_j)= 0 \text{ and } \lambda(C_i\cap S_{\beta}C_j)= 0.
\end{array}\right.
\end{equation}
\noindent Define the partition $\mathcal P$ of $\Omega \times \bint$ by
\[ \mathcal P=\big\{\Omega \times C_j : j \not \in s \big\} \cup \big\{ \{\omega_1=i\} \times C_j: i \in \{0,1\}, \, j \in s \big\}.\]
Then $\mathcal P$ is a Markov partition underlying the map $K$. Let $Y$ denote the topological Markov chain determined by the matrix $A$. That is, $Y=\{(y_n)_{n \ge 1}\in \{1,\dots ,L\}^{\mathbb N}:a_{y_n, y_{n+1}}=1 \}$. Let $\mathcal Y$ denote the $\sigma$-algebra on $Y$ determined by the cylinder sets, i.e., the sets specifying finitely many digits, and let $\sigma_Y$ be the left shift on $Y$. We use Parry's recipe (\cite{Par64}) to determine the Markov measure $Q$ of maximal entropy for $(Y, \mathcal Y, \sigma_Y)$. By results in \cite{DV05} we know that $\nu_{\beta}$ is the unique measure of maximal entropy for $K$ with entropy $h_{\nu_{\beta}}(K)=\log \lceil \beta \rceil$. By the identification with the Markov chain we know that $h_Q(\sigma_Y) = \log \lceil \beta \rceil$. One then gets that the corresponding transition matrix $(p_{i,j})$ for $Y$ satisfies $p_{i,j}=a_{i,j}\frac{v_j}{\lceil \beta \rceil v_i}$, where $(v_1,v_2, \ldots,v_L)$ is the right probability eigenvector of $A$ with eigenvalue $\lfloor \beta \rfloor +1$. From this we see that if $[j_1 \cdots j_m]$ is an allowed cylinder in $Y$, then
\begin{equation}\label{q:qmeasure}
Q([j_1 \cdots j_m])=\frac{v_{j_m}}{\lceil \beta \rceil^{m-1}}.
\end{equation}
Property {\bf p5} from \cite{DV05} says that for all $i \in s$, $a_{i,1}=a_{i,L}=1$ and $a_{i,j}=0$ for all other $j$. By symmetry of the matrix $(p_{i,j})$, it follows that
\begin{equation}\label{q:ps1L}
p_{i, 1} = p_{i,L} = \frac12 \quad \text{for all } i \in s.
\end{equation}
Let
\[ X = \Omega \times \bint \backslash \big( \bigcup_{n \ge 0} K^{-n} F  \big).\]
Then $\nu_{\beta} (X) =1$. The isomorphism $\alpha: X \to Y$ between $(K, \nu_{\beta})$ and $(\sigma_Y, Q)$ is then given by 
$$\alpha_j(\omega, x) = k \mbox{ if } K^{j-1}(\omega, x) \in C_k.$$ 
See Theorem 7 in \cite{DV05} for a proof of this fact. In Figure~\ref{f:diagram} we see the relation between the different systems we have introduced so far.
\begin{figure}[ht]
\centering
\includegraphics{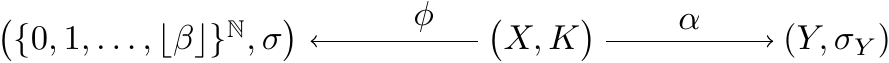}
\caption{The relation between the different spaces.}
\label{f:diagram}
\end{figure}

To study the local dimension of $\nu_{\beta}$, we need to consider balls in $X$ under a suitable metric. Define the metric $\rho$ on $X$ by
\[ \rho\big((\omega,x),(\omega^{\prime},x^{\prime})\big)=\beta^{-\min\{k\ge 0\, : \, \omega_{k+1}\not=\omega^{\prime}_{k+1} \text{ or } \alpha_{k+1}(\omega,x)\not=\alpha_{k+1}(\omega^{\prime},x^{\prime}) \}}.\]
Consider the ball 
\[ B_{\rho}\big((\omega,x),\beta^{-k}\big)=\{(\omega^{\prime},x^{\prime})\, : \, \omega_i^{\prime}=\omega_i, \mbox{ and } \, \alpha_i(\omega^{\prime},x^{\prime})= \alpha_i(\omega,x), \, i=1,\cdots ,k\}.\]
Let
\[ M_k(\omega,x)=\sum_{i=0}^{k-1}{\bf 1}_{X \cap \Omega\times S}\big(K^i(\omega,x)\big) = \# \{ 1 \le i \le k \, : \, \alpha_i(\omega,x) \in s \}.\]
To determine $\nu_{\beta}\big(B_{\rho}((\omega,x),r)\big)$ for $r \downarrow 0$ we will calculate $Q \Big( \alpha \Big(B_{\rho}\big((\omega,x),\beta^{-k}\big) \Big)\Big)$. For all points $(\omega',x')$ in the ball $B_{\rho}\big((\omega,x),\beta^{-k}\big)$ the $\alpha$-coding starts with $\alpha_1(\omega, x) \cdots \alpha_k(\omega, x)$ and $\omega'$ starts with $\omega_1 \cdots \omega_k$. From the second part we know what happens the first $k$ times that the $K$-orbit of a point $(\omega', x')$ lands in $\Omega \times S$. Since $M_k(\omega, x)$ of these values have been used for $\alpha_1(\omega, x)\cdots \alpha_k(\omega,x)$, there are $k-M_k(\omega,x)$ unused values left. Note that $M_k(\omega', x') = M_k(\omega, x)=M_k$ for all $(\omega', x') \in B_{\rho}\big((\omega,x),\beta^{-k}\big)$. Define the set
\[ Z = X \cap \bigcap_{n \ge 1}\bigcup_{i \ge 1}  K^{-i} \big(\Omega \times S\big).\]
All points in $Z$ land in the set $\Omega \times S$ infinitely often under $K$. Since $Z$ is $K$-invariant, by ergodicity of $K$ we have $\nu_{\beta} (Z) =1$. So, all points $(\omega', x') \in B_{\rho}\big((\omega,x),\beta^{-k}\big) \cap Z$ make a transition to $S$ some time after $k$. Moreover, after this transition these points move to $C_1$ if $\omega_{M_k+1} =1$ and to $C_L$ otherwise. The image of a point $(\omega', x') \in B_{\rho}\big((\omega,x),\beta^{-k}\big)$ under $\alpha$ will thus have the form
\begin{multline*}
\alpha_1 \cdots \alpha_k \, \underbrace{a_{k+1} \cdots a_{m_1-1}}_{\not \in s} \, \underbrace{a_{m_1}}_{\in s} \, \underbrace{a_{m_1+1} \cdots a_{m_2-1}}_{\not \in s} \, \underbrace{a_{m_2}}_{\in s} \\
\cdots \, \underbrace{a_{m_{N-1}-1}\cdots a_{m_N-1}}_{\not \in s} \, \underbrace{a_{m_N}}_{\in s} \, \underbrace{a_{m_N+1}a_{m_N+2} \cdots}_{\text{tail}},
\end{multline*}
where $a_{m_j+1} \in \{1,L\}$, $m_{j+1}-m_j-2 \ge 1$ and $N = k-M_k(\omega, x)$. Note that by the ergodicity of $\nu_{\beta}$ we have
\[ Q \Big( \bigcup_{m \ge 1} [a_1 \cdots a_m] \, : \, a_m \in s \text{ and } a_i \not \in s , \, i < m \Big) = \nu_{\beta} \Big( X \cap \bigcup_{m \ge 1} K^{-m} (\Omega \times S) \Big)=1.\]
So, the transition from any state to $s$ occurs with probability 1. Then one of the digits $\omega_j$, $M_k+1 \le j \le k$, specifies what happens in this event and by (\ref{q:ps1L}) both possibilities happen with probability $\frac12$. To determine the measure of all possible tails of sequences in $\alpha \Big(B_{\rho}\big((\omega,x),\beta^{-k}\big) \Big)$, note that again by {\bf p5} of \cite{DV05} this tail always belongs to a point in $\Omega \times C_1$ or $\Omega \times C_L$. Since the $\nu_{\beta}$-measure of these sets is the same, the $Q$-measure of the set of all possible tails is given by $\nu_{\beta}(\Omega \times C_1)= \mu_{\beta}(C_1)$. Putting all this together gives
\begin{equation}\label{q:Qball}
Q \Big( \alpha \Big(B_{\rho}\big((\omega,x),\beta^{-k}\big) \Big)\Big) = \lceil \beta \rceil^{-(k-1)} v_{\alpha_k(\omega,x)} \cdot \underbrace{1 \cdot \frac{1}{2} \cdot 1 \cdot \frac{1}{2} \cdots 1\cdot \frac{1}{2}}_{k-M_k(\omega, x) \text{ times}} \cdot \, \mu_{\beta}(C_1),
\end{equation}
and hence,
\begin{equation}\label{q:nuball}
 \nu_{\beta}\big(B_{\rho}((\omega,x),\beta^{-k})\big) = \lceil \beta \rceil^{-(k-1)}v_{\alpha_k(\omega,x)}\,2^{-(k-M_k(\omega,x))}\,\mu_{\beta}(C_1).
\end{equation}
This gives the following theorem.
\begin{thm}\label{t:locdim}
Let $\beta >1$ be a generalised multinacci number. For all $(\omega,x) \in X$ we have
\begin{multline}\label{q:locdim}
\frac{\log \lceil \beta \rceil}{\log \beta} +\frac{\log 2}{\log \beta}\Big[ 1-\limsup_{k\to\infty}\frac{M_k(\omega,x)}{k} \Big] \le \underline d\big(\nu_{\beta}, (\omega,x)\big)\\
\le \overline d\big(\nu_{\beta}, (\omega,x)\big) \le \frac{\log \lceil \beta \rceil}{\log \beta} +\frac{\log 2}{\log \beta}\Big[ 1-\liminf_{k\to\infty}\frac{M_k(\omega,x)}{k} \Big].
\end{multline}
\end{thm}

\begin{proof}
Let $\frac{1}{\beta^{k+1}} < r \le \frac{1}{\beta^k}$. Set $v_{min} = \min\{v_1, \ldots, v_L\}$ and $v_{max} = \max\{v_1, \ldots, v_L\}$. Then, by (\ref{q:nuball}),
\[ \frac{\log \nu_{\beta}\big(B_{\rho}((\omega,x),r)\big)}{\log r} \le  \frac{(k-1)\log \lceil \beta \rceil}{k \log \beta} + \frac{(k-M_k(\omega, x))\log 2}{k\log\beta} - \frac{\log \big(v_{min}\, \mu_{\beta}(C_1)\big)}{k \log \beta}.\]
Hence,
\[ \overline d\big(\nu_{\beta}, (\omega,x)\big) = \limsup_{k \to \infty} \frac{\log \nu_{\beta}\big(B_{\rho}((\omega,x),r)\big)}{\log r} \le \frac{\log \lceil \beta \rceil}{\log \beta} + \frac{\log 2}{\log \beta}\Big[ 1 - \liminf_{k \to \infty} \frac{M_k(\omega,x)}{k} \Big].\]
On the other hand,
\[ \frac{\log \nu_{\beta}\big(B_{\rho}((\omega,x),r)\big)}{\log r} \ge \frac{(k-1)\log \lceil \beta \rceil}{(k+1) \log \beta} + \frac{(k-M_k(\omega, x))\log 2}{(k+1)\log\beta} - \frac{\log \big(v_{max}\, \mu_{\beta}(C_1)\big)}{(k+1) \log \beta}.\]
Since $M_{k+1}(\omega, x)-1 \le M_k(\omega,x) \le M_{k+1}(\omega, x)$, we have that
\[ \underline d\big(\nu_{\beta}, (\omega,x)\big) = \liminf_{k \to \infty} \frac{\log \nu_{\beta}\big(B_{\rho}((\omega,x),r)\big)}{\log r} \ge \frac{\log \lceil \beta \rceil}{\log \beta} + \frac{\log 2}{\log \beta}\Big[ 1 - \limsup_{k \to \infty} \frac{M_k(\omega,x)}{k} \Big].\]
This proves the theorem.
\end{proof}

\begin{rem}{\rm From the proof of the previous theorem it follows that if $\lim_{k\to\infty}\frac{M_k(\omega,x)}{k}$ exists, then $d\big(\nu_{\beta}, (\omega,x) \big)$ exists and is equal to $\frac{\log \lceil \beta \rceil}{\log \beta} +\frac{\log 2}{\log \beta}\Big[ 1-\lim_{k\to\infty}\frac{M_k(\omega,x)}{k}\Big]$.
}\end{rem}

\begin{cor}\label{c:locdimnubeta}
Let $\beta$ be a generalised multinacci number. The local dimension function $d\big(\nu_{\beta}, (\omega,x)\big)$ is constant $\nu_{\beta}$-a.e.~and equal to
\[d\big(\nu_{\beta}, (\omega,x)\big)=\frac{\log \lceil \beta \rceil}{\log \beta} +\frac{\log 2}{\log \beta} \big( 1-\mu_{\beta}(S) \big).\]
\end{cor}

\begin{proof}
Since $\nu_{\beta}$ is ergodic, the Ergodic Theorem gives that for $\nu_{\beta}$-a.e.~$(\omega, x)$,
\[ \lim_{k \to \infty} \frac{M_k(\omega, x)}{k} = \nu_{\beta}\big(\Omega \times S\big) = \mu_{\beta}(S).\]
This gives the result.
\end{proof}

\medskip
Recall that $\phi$ maps points $(\omega, x)$ to digit sequences $\big( b_n(\omega,x) \big)_{n \ge 1}$. It is easy to see that $x = \sum_{n \ge 1} \frac{b_n(\omega,x)}{\beta^n}$ for each choice of $\omega \in \Omega$. Note that a point $x$ has exactly one $\beta$-expansion if and only if for all $n \ge 0$, $\pi\big( K^n (\omega,x) \big) \not \in S$. Let $\mathcal A_{\beta} \subset \bint$ be the set of points with a unique $\beta$-expansion. Then $\mathcal A_{\beta} \neq \emptyset$, since $0, \fb \in \mathcal A_{\beta}$ for any $\beta>1$. By Proposition 2 from \cite{DV05} all elements from $\cup_{n \ge 0}K^{-n}F$ will be in $S$ at some point and hence they will have more than one expansion. So, $\mathcal A_{\beta} \subset X$. The next result also follows easily from Theorem~\ref{t:locdim}. The measure $\nu_{\beta}$ is called {\bf multifractal} if the local dimension takes more than one value on positive Hausdorff dimension sets. 

\begin{cor}
Let $\beta$ be a generalised multinacci number. If $x \in \mathcal A_{\beta}$, then $d\big(\nu_{\beta}, (\omega,x)\big) =\frac{\log \lceil \beta \rceil + \log 2}{\log \beta}$ for all $\omega \in \Omega$. The measure $\nu_{\beta}$ is multifractal.
\end{cor}

\begin{proof}
If $x \in \mathcal A_{\beta}$, then $M_k(\omega, x)=0$ for all $\omega \in \Omega$ and $k \ge 1$. Hence, by (\ref{q:locdim}), $d\big(\nu_{\beta}, (\omega,x)\big) =\frac{\log \lceil \beta \rceil + \log 2}{\log \beta}$. From standard results in dimension theory and our choice of metric it follows that $dim_H\big(\Omega \times \{x\}\big)=\frac{\log 2}{\log \beta}$ for all $x\in \mathcal A_{\beta}$. Hence, $\nu_{\beta}$ is a multifractal measure.
\end{proof}

\begin{exa}\label{r:goldenmean}
We give a example to show what can happen on points in $F$. Let $\beta =\frac{1+\sqrt 5}{2}$ be the golden ratio. Then, $1 = \frac{1}{\beta} + \frac{1}{\beta^2}$. Figure~\ref{f:gm} shows the maps $T_0$ and $T_1$ for this $\beta$.
\begin{figure}[ht]
\centering
\includegraphics{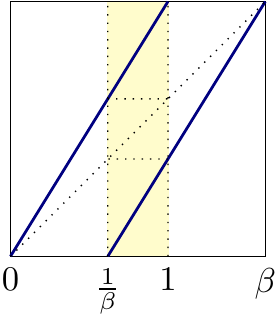}
\caption{The maps $T_0 x = \beta x$ and $T_1 x = \beta x-1$ for $\beta = \frac{1+\sqrt 5}{2}$. The region $S$ is colored yellow.}
\label{f:gm}
\end{figure}
Note that $F = \{ 0, \frac{1}{\beta}, 1, \beta \}$. The partition $\mathcal C$ consists of only three elements and the transition matrix and stationary distribution of the corresponding Markov chain are
\[ P = \left(
\begin{array}{ccc}
1/2 & 1/2 & 0\\
1/2 & 0 & 1/2\\
0 & 1/2 & 1/2
\end{array} \right) \quad \text{ and } \quad v=(1/3,1/3,1/3).\]
Hence, $\mu_{\beta}(S)=1/3$ and Corollary~\ref{c:locdimnubeta} gives that for $\nu_{\beta}$-a.e.~$(\omega,x) \in \Omega \times [0, \beta]$,
\[d\big(\nu_{\beta},(\omega,x)\big)=\frac{\log 2}{\log \beta} \Big[ 2-\mu_{\beta}(S) \Big]=(2-1/3) \frac{\log 2}{\log \beta}=\frac{5\log 2}{6 \log \beta}.\]

\vskip .2cm
Now consider the $\alpha$-code of the points $(\overline{10},1)$ and $(\overline{01},1/\beta)$, where the bar indicates a repeating block:
\[\alpha \big((\overline{10},1)\big)=\alpha \big((\overline{01},1/\beta)\big)=(s,s,s,\cdots),\]
which is not allowed in the Markov chain $Y$. Then for any point 
\[(\omega,x)\in \bigcup_{m=0}^{\infty} K^{-m} \big(\{(\overline{10},1),(\overline{01},1/\beta)\}\big),\]
one has $B_{\rho}\big((\omega,x),\beta^{-k}\big)$ is a countable set for all $k$ sufficiently large. For the local dimension this implies that
\[ d\big(\nu_{\beta},(\omega,x)\big) = \lim_{r \downarrow 0} \frac{\log 0}{\log r} = \infty.\]
\end{exa}

\section{Local dimensions for the symmetric Bernoulli convolution}

Consider now the Bernoulli convolution measure $\mu_{\beta}=\nu_{\beta}\circ \pi^{-1}$ on the interval $\bint$ for generalised multinacci numbers. In \cite{FS11}, the local dimension of $\mu_{\beta}$ with respect to the Euclidean metric was obtained for all Pisot numbers $\beta$. A {\bf Pisot number} is an algebraic integer that has all its Galois conjugates inside the unit circle. It is well known that all multinacci numbers are Pisot numbers, but unfortunately not all generalised multinacci numbers are Pisot. In Remark~\ref{r:pisot}(i) we list some classes of generalised multinacci numbers that are in fact Pisot numbers. Before stating the results from \cite{FS11}, we introduce more notation. Let
\begin{equation}\label{q:fs}
\mathcal N_k (x,\beta)=\#\left\{(a_1, \ldots, a_k)\in \{0,1, \ldots, \lfloor \beta \rfloor\}^k:\exists (a_{k+n})_{n \ge 1} \text{ with } x=\sum_{m=1}^{\infty} \frac{a_m}{\beta^m} \right\}.
\end{equation}
A straightforward calculation (see also Lemma 4.1 of \cite{Kem12}) shows  that
\[ \mathcal N_k (x,\beta)=\int_{\Omega} \, 2^{M_k(\omega,x)} \, dm(\omega),\]
where $m$ is the uniform Bernoulli measure on $\{0,1, \ldots, \lfloor \beta \rfloor\}^{\mathbb N}$ as before. In \cite{FS11}, it was shown that if $\beta$ is a Pisot number, then there is a constant $\gamma=\gamma(\beta,m)$ such that 
\begin{equation}\label{q:gamma}
\lim_{k \to\infty}\frac{\log \mathcal N_k(x,\beta)}{k}=\gamma
\end{equation}
for $\lambda$-a.e.~$x$ in $\bint$, where $\lambda$ is the one-dimensional Lebesgue measure. Using this, Feng and Sidorov obtained that for $\lambda$-a.e.~$x$,
\[ d(\mu_{\beta},x)=\frac{\log \lceil \beta \rceil -\gamma}{\log \beta}\]
 In fact, the result they obtained was stronger, but we will use their result in this form. We will show that one has the same value for the local dimension when the Euclidean metric on $\mathbb R$ is replaced by the Hausdorff metric. To this end, consider the metric $\bar \rho$ on $\bint$ defined by
\[ \bar \rho(x,y)=d_H\big(\pi^{-1}(x),\pi^{-1}(y)\big),\]
where $d_H$ is the Hausdorff distance given by
\begin{multline*}
d_H \big(\pi^{-1}(x),\pi^{-1}(y)\big)\\
=\inf\Big\{\epsilon >0:\pi^{-1}(y)\subset \bigcup_{\omega \in \Omega}B_{\rho}\big((\omega,x),\epsilon \big) \mbox{ and } \pi^{-1}(x)\subset \bigcup_{\omega \in \Omega}B_{\rho}\big((\omega,y),\epsilon\big)\Big\}.
\end{multline*}

\begin{thm}
Let $\beta$ be a generalised multinacci number. Then for all $x \in \bint$,
\begin{multline*}
\frac{1}{\log \beta} \Big[ \log \lceil \beta \rceil - \limsup_{k \to\infty}\frac{\log \mathcal N_k(x,\beta)}{k} \Big] \le \underline d(\mu_{\beta},x)\\
\le \overline d(\mu_{\beta},x) \le \frac{1}{\log \beta} \Big[ \log \lceil \beta \rceil - \liminf_{k \to\infty}\frac{\log \mathcal N_k(x,\beta)}{k} \Big].
\end{multline*}
\end{thm}

\begin{proof}
Let $B_ {\bar \rho}(x,\epsilon)=\{y:\bar \rho(x,y)<\epsilon\}$. We want to determine explicitly the set $\pi^{-1}\big(B_{\bar \rho}(x,\beta^{-k}) \big)$. First note that for any $(\omega,x)$, and any $k\ge 0$, one has $(\omega^{\prime},y)\in B_{\rho}\big((\omega,x),\beta^{-k}\big)$, and $B_{\rho}\big((\omega,x),\beta^{-k}\big)=B_{\rho}\big((\omega^{\prime},y),\beta^{-k}\big)$ for any $(\omega^{\prime},y)\in [\omega_1\cdots\omega_k]\times [\alpha_1(\omega,x)\cdots \alpha_k(\omega,x)]$. We denote the common set by $B_{\rho}\big(([\omega_1 \cdots \omega_k],x),\beta^{-k}\big)$. This implies that
\[\pi^{-1}\big(B_{\bar \rho}(x,\beta^{-k})\big)=\bigcup_{[\omega_1\cdots \omega_k]}B_{\rho}\big(([\omega_1 \cdots \omega_k],x),\beta^{-k}\big),\]
where the summation on the right is taken over all possible cylinders of length $k$ in $\Omega$. Again set $v_{min} = \min\{v_1, \ldots, v_L\}$ and $v_{max} = \max\{v_1, \ldots, v_L\}$. Then,
\begin{eqnarray*}
\mu_{\beta}\big(B_{\bar \rho}(x,\beta^{-k})\big) &=& \sum_{[\omega_1 \cdots \omega_k]}\nu_{\beta}\big(B_{\rho}\big(([\omega_1 \cdots \omega_k],x),\beta^{-k} \big)\big)\\
& \le & \sum_{[\omega_1 \cdots \omega_k]}\lceil \beta \rceil^{-(k-1)}v_{\alpha_k(\omega,x)}\,2^{-(k-M_k(\omega,x))}\, \mu_{\beta}(C_1)\\
& \le & \lceil \beta \rceil^{-(k-1)}v_{max} \, \mu_{\beta}(C_1) \sum_{[\omega_1 \cdots \omega_k]} 2^{M_k(\omega,x)} 2^{-k}\\
& = & \lceil \beta \rceil^{-(k-1)}v_{max}\, \mu_{\beta}(C_1) \int_{\Omega} \, 2^{M_k(\omega,x)} \, dm(\omega)\\
& = & \lceil \beta \rceil^{-(k-1)} \, v_{max}\,   \mu_{\beta}(C_1) \, {\mathcal N}_k(x,\beta).
\end{eqnarray*}
Now taking logarithms, dividing by $\log \beta^{-k}$, and taking limits we get
\[ \underline d(\mu_{\beta},x)\ge \frac{\log \lceil \beta \rceil}{\log \beta} - \frac{1}{\log \beta} \limsup_{k \to \infty} \frac{\log \mathcal N_k(x, \beta)}{k}.\]
Similarly we get that
\[ \mu_{\beta}\big(B_{\bar \rho}(x,\beta^{-k})\big) \ge \lceil \beta \rceil^{-(k-1)} \, v_{min} \, \mu_{\beta}(C_1)  \, {\mathcal N}_{k}(x,\beta),\]
which gives
\[ \overline d(\mu_{\beta},x)\le \frac{\log \lceil \beta \rceil}{\log \beta} - \frac{1}{\log \beta} \liminf_{k \to \infty} \frac{\log \mathcal N_k(x, \beta)}{k}. \qedhere\]
\end{proof}

\noindent By the results from \cite{FS11} we have the following corollary.
\begin{cor}
If $\beta$ is Pisot, then $d(\mu_{\beta},x)$ exists for $\lambda$-a.e.~$x$ and is equal to $\frac{\log \lceil \beta \rceil- \gamma}{\log \beta}$, where $\gamma$ is the constant from (\ref{q:gamma}).
\end{cor}

\begin{rem}\label{r:pisot}{\rm
(i) We give some examples of generalised multinacci numbers that are Pisot numbers. The generalised multinacci numbers in the interval $[1,2]$, i.e., the multinacci numbers, are all Pisot. In the interval $[2, \infty)$ the numbers $\beta$ that satisfy $\beta^2 - k\beta-1=0$, $k \ge 2$, are all Pisot as well. Recall that if $1=\frac{a_1}{\beta} + \cdots + \frac{a_n}{\beta^n}$ with $a_i \ge 1$ for all $1 \le i \le n$, then $n$ is called the degree of $\beta$. From Theorem 4.2 in \cite{AG05} by Akiyama and Gjini we can deduce that all generalised multinacci numbers of degree 3 are Pisot numbers. Similarly, from Proposition 4.1 in \cite{AG05} it follows that all generalised multinacci numbers of degree 4 with $a_4=1$ are Pisot. An example of a generalised multinacci number that is not Pisot is the number $\beta$ satisfying
\[ 1=\frac{3}{\beta} + \frac{1}{\beta^2} + \frac{2}{\beta^3} + \frac{3}{\beta^4}.\]
(ii) In \cite{Kem12} it is shown that for all $\beta > 1$ and $\lambda$-a.e.~$x$, $\liminf_{k \to \infty} \frac{\log \mathcal N_k(x, \beta)}{k} \ge \mu_{\beta}(S)\log 2$, so we get
\[ \overline d(\mu_{\beta},x)\le \frac{1}{\log \beta}\Big[ \log \lceil \beta \rceil - \mu_{\beta}(S)\log 2 \Big].\]
Kempton also remarks that this lower bound is not sharp.
}\end{rem}

\section{Asymmetric random $\beta$-transformation: the golden ratio}

In the previous section, we considered the measure $\nu_{\beta}=\nu_{\beta,1/2}$ which is the lift of the uniform Bernoulli measure $m=m_{1/2}$ under the isomorphism $\phi(\omega, x)=\big(b_n(\omega,x)\big)_{n \ge 1}$. The projection of $\nu_{\beta}$ in the second coordinate is the symmetric Bernoulli convolution. In this section, we will investigate the asymmetric Bernoulli convolution in case $\beta$ is the golden ratio.

\medskip
Let $\beta = \frac{1+\sqrt 5}{2}$ and consider the $(p,1-p)$-Bernoulli measure $m_p$ on $\{0,1\}^{\mathbb N}$, i.e., with $m_p([0])=p$ and $m_p([1])=1-p$. Let $\nu_{\beta,p}=m_p\circ \phi$ on $\Omega \times [0, \beta]$. Since $m_p$ is shift invariant and ergodic, we have that $\nu_{\beta,p}$ is $K$-invariant and ergodic. We first show that $\nu_{\beta,p}$ is a Markov measure with the same Markov partition as in the symmetric case (see Example~\ref{r:goldenmean}), but the transition probabilities as well as the stationary distribution are different. This is achieved by looking at the $\alpha$-code as well. The Markov partition is given by the partition $\{E_0,S,E_1\}$, and the corresponding Markov chain has three states $\{e_0,s,e_1\}$ with transition matrix
\[ P_p=\left( \begin{array}{ccc}
p & 1-p & 0 \\
p & 0 & 1-p \\
0 & p & 1-p \end{array} \right)\] 
and stationary distribution 
\[ u=(u_{e_0},u_s,u_{e_1})=\Big(\frac{p^2}{p^2-p+1},\frac{p(1-p)}{p^2-p+1},\frac{(1-p)^2}{p^2-p+1} \Big).\]
We denote the corresponding Markov measure by $Q_p$, that is
\[ Q_p\big([j_1 \cdots j_k]\big)=u_{j_1}p_{j_1,j_2}\cdots p_{j_k,j_{k+1}},\]
and the space of realizations by 
\[ Y=\big\{(y_1,y_2,\ldots):y_i\in \{e_0,s,e_1\}, \mbox{ and } p_{y_i,y_{i+1}}>0\big\}.\]
Consider the map $\alpha:\Omega \times [0,\beta]$ of the previous section, namely
\[ \alpha_j(\omega,x) = \left\{ \begin{array}{ll}
         e_0, & \mbox{if $K^{j-1}(\omega,x)\in \Omega \times E_0$};\\
        s, & \mbox{if $K^{j-1}(\omega,x)\in \Omega \times S$};\\
e_1, & \mbox{if $K^{j-1}(\omega,x)\in \Omega \times E_1$}.\end{array} \right. \] 
Define $\psi:Y\to \{0,1\}^{\mathbb N}$ by
\[ \psi(y)_j = \left\{ \begin{array}{ll}
         0, & \mbox{if $y_j=e_0$ or $y_jy_{j+1}=se_1$};\\
         1, & \mbox{if $y_j=e_1$ or $y_jy_{j+1}=se_0$}.\end{array} \right. \] 
It is easy to see that $\psi\circ \alpha=\phi$. We want to show that $Q_p\circ \alpha =\nu_{\beta,p}$. Since $\nu_{\beta,p}=m_p\circ \phi$, we show instead the following.

\begin{prop}
We have $m_p=Q_p\circ \psi^{-1}$.
\end{prop}

\begin{proof}
It is enough to check equality on cylinders. To avoid confusion, we denote cylinders in $\{0,1\}^{\mathbb N}$ by $[i_1 \cdots i_k]$ and cylinders in $Y$ by $[j_1 \cdots j_k]$. We show by induction that
\[ \psi^{-1}([i_1 \cdots i_k])=[j_1 \cdots j_k]\cup [j^{\prime}_1,\ldots, j^{\prime}_{k+1}],\]
where $j_k=e_{i_k}$, and $j^{\prime}_kj^{\prime}_{k+1}=se_{1-i_k}$, and
\[ Q_p([j_1 \cdots j_k])+Q_p([j^{\prime}_1 \cdots  j^{\prime}_{k+1}])=m_p\big([i_1 \cdots i_k]\big) =p^{k-\sum_{\ell=1}^k i_{\ell} }(1-p)^{\sum_{\ell=1}^k i_{\ell}}.\]
Consider the case $k=1$. We have $\psi^{-1}[0]=[e_0]\cup [se_1]$ and $\psi^{-1}[1]=[e_1]\cup [se_0]$. Furthermore,
\[ Q_p([e_0]\big)+Q_p([se_1])=\frac{p^2}{p^2-p+1}+\frac{p(1-p)^2}{p^2-p+1}=p=m_p([0]),\]
and
\[ Q_p([e_1])+Q_p([se_0])=\frac{(1-p)^2}{p^2-p+1}+\frac{p^2(1-p)}{p^2-p+1}=1-p=m_p([1])\]
as required.  Assume now the result is true for all cylinders $[i_1 \cdots i_k]$ of length $k$, and consider a cylinder $[i_1 \cdots  i_{k+1}]$ of length $k+1$. Then,
\[ \psi^{-1}([i_1 \cdots i_{k+1}])=[j_1 \cdots j_{k+1}]\cup [j^{\prime}_1 \cdots j^{\prime}_{k+2}],\]
where
\[ \psi^{-1}([i_2 \cdots i_{k+1}])=[j_2 \cdots j_{k+1}]\cup [j^{\prime}_2 \cdots j^{\prime}_{k+2}],\]
and 
\[ j_1, j^{\prime}_1= \left\{ \begin{array}{ll}
         e_{i_1}, & \mbox{if $i_1=i_2$ or $i_1\not= i_2$ and $j_2=s$},\\
         s, & \mbox{if $i_1\not= i_2$ and $j_2\not=s$},\end{array} \right. \] 
By the definition of $Q_p$, we have 
\[ Q_p([j_1 \cdots j_{k+1}])=\frac{u_{j_1} p_{j_1,j_2}}{u_{j_2}} Q_p([j_2 \cdots j_{k+1}]),\]
and
\[ Q_p([j^{\prime}_1 \cdots j^{\prime}_{k+2}])=\frac{u_{j^{\prime}_1} p_{j^{\prime}_1,j^{\prime}_2}}{u_{j^{\prime}_2}} Q_p([j^{\prime}_2 \cdots j^{\prime}_{k+2}]).\]
One easily checks that
\[ \frac{u_{j_1} p_{j_1,j_2}}{u_{j_2}}=\frac{u_{j^{\prime}_1} p_{j^{\prime}_1,j^{\prime}_2}}{u_{j^{\prime}_2}}=p^{1-i_1}(1-p)^{i_1} =\left\{ \begin{array}{ll}
         p, & \mbox{if }i_1=0;\\
        1-p, & \mbox{if }i_1=1.\end{array} \right. \] 
By the induction hypothesis applied to the cylinder $[i_2 \cdots i_{k+1}]$, we have $j_{k+1}=e_{i_{k+1}}$, and $j^{\prime}_{k+1}j^{\prime}_{k+2}=se_{1-i_{k+1}}$, and
\[ Q_p([j_2 \cdots j_{k+1}])+Q_p([j^{\prime}_2 \cdots j^{\prime}_{k+2}])=m_p([i_2\cdots i_{k+1}])
=p^{k-\sum_{\ell=2}^{k+1} i_{\ell} }(1-p)^{\sum_{\ell=2}^{k+1} i_{\ell}}.\]
Thus,
\begin{eqnarray*}
Q_p([j_1 \cdots j_{k+1}])+Q_p([j^{\prime}_1 \cdots j^{\prime}_{k+2}]) & = &
p^{1-i_1}(1-p)^{i_1} p^{k-\sum_{\ell=2}^{k+1} i_{\ell} }(1-p)^{\sum_{\ell=2}^{k+1} i_{\ell}}\\
& = & p^{k+1-\sum_{\ell=1}^{k+1} i_{\ell} }(1-p)^{\sum_{\ell=1}^{k+1} i_{\ell}}\\
& = & m_p([i_1\cdots i_{k+1}]).
\end{eqnarray*}
This gives the result.
\end{proof}

\noindent As before, let $M_k(\omega,x)=\sum_{i=0}^{k-1}{\bf 1}_{\Omega \times S}(K_{\beta}^i(\omega,x))$.
\begin{thm} For $\nu_{\beta,p}$-a.e.~$(\omega,x)$ for which $\displaystyle\lim_{k\to\infty}\frac{M_k(\omega,x)}{k}$ exists, one has
\[ d\big( \nu_{\beta,p},(\omega, x)\big)=\frac{H(p)}{\log \beta} \Big(2-\lim_{k\to\infty}\frac{M_k(\omega,x)}{k}\Big),\]
where $H(p)=-p\log p -(1-p)\log(1-p)$.
\end{thm}

\begin{proof} We consider the same metric $\rho$ as in the previous section, namely
\[ \rho\big((\omega,x),(\omega^{\prime},x^{\prime})\big)=\beta^{-\min\{k\ge 0 \, : \, 
\omega_{k+1}\neq \omega^{\prime}_{k+1} \text{ or } \alpha_{k+1}(\omega,x)\neq \alpha_{k+1}(\omega^{\prime},x^{\prime})\}}.\]
Consider a point $(\omega,x)$ such that $\displaystyle\lim_{k\to\infty}\frac{M_k(\omega,x)}{k}$ exists. By the same reasoning that led to (\ref{q:Qball}) we have
\begin{multline*}
\nu_{\beta,p}\big(B_{\rho}((\omega,x),\beta^{-k})\big)\\
= Q_p([\alpha_1(\omega,x),\ldots,\alpha_k(\omega,x)]) p^{(k-M_k(\omega,x))-\sum_{i=M_k(\omega,x)+1}^k \omega_i} (1-p)^{\sum_{i=M_k(\omega,x)+1}^k \omega_i} u_{e_{1-\omega_k}}.
\end{multline*}
Let $u_{\max}=\max(u_{e_0},u_{e_1})$ and $u_{\min}=\min(u_{e_0},u_{e_1})$, then $\log \nu_{\beta,p}\big(B_{\rho}((\omega,x),\beta^{-k})\big)$ is bounded from above by
\begin{gather*}
\log Q_p([\alpha_1(\omega,x),\ldots,\alpha_k(\omega,x)]) + \Big((k-M_k(\omega,x)-\sum_{i=M_k(\omega,x)+1}^k \omega_i\Big)\log p\\
 + \sum_{i=M_k(\omega,x)+1}^k \omega_i\log (1-p) +\log u_{\max},
\end{gather*}
and is bounded from below by
\begin{gather*}
\log Q_p([\alpha_1(\omega,x),\ldots,\alpha_k(\omega,x)]) + \Big((k-M_k(\omega,x)-\sum_{i=M_k(\omega,x)+1}^k \omega_i\Big)\log p\\
+ \sum_{i=M_k(\omega,x)+1}^k \omega_i\log (1-p) +\log u_{\min}.
\end{gather*}
Dividing by $-k\log \beta$ and taking limits, we have by the Shannon-McMillan-Breiman Theorem that
\[ \lim_{k\to\infty}\frac{\log Q_p([\alpha_1(\omega,x) \cdots \alpha_k(\omega,x)])}{-k\log \beta}= \frac{H(p)}{\log \beta},\]
and by the Ergodic Theorem we have
\[ \lim_{k\to\infty}\frac{\sum_{i=M_k(\omega,x)+1}^k \omega_i}{-k\log \beta}= \frac{-(1-p)\big(1-\lim_{k\to\infty}\frac{M_k(\omega,x)}{k}\big)}{\log \beta},\]
both for $\nu_{\beta}$-a.e.~$(\omega,x)$. Thus, both the upper and the lower bounds converge to the same value, implying that
\[ d\big(\nu_{\beta,p}, (\omega, x)\big)=\frac{H(p)}{\log \beta} \Big(2-\lim_{k\to\infty}\frac{M_k(\omega,x)}{k}\Big). \qedhere\]
\end{proof}

\begin{cor}
For $\nu_{\beta,p}$-a.e.~$(\omega,x)$ one has 
\[ d\big(\nu_{\beta,p}, (\omega, x)\big)=\frac{H(p)}{\log \beta} \Big(2-\nu_{\beta,p}\big(\Omega \times S\big)\Big) =\frac{H(p)}{\log \beta} \Big(2-\frac{p(1-p)}{p^2-p+1}\Big).\]
\end{cor}

\vskip .5cm
We now turn to the study of the local dimension of the asymmetric Bernoulli convolution $\mu_{\beta,p}$, which is the projection in the second coordinate of $\nu_{\beta,p}$. Let ${\mathcal N}_k(\omega,x)$ be as given in equation (\ref{q:fs}). In the symmetric case, it was shown that
\[\mathcal N_{k}(x,\beta)=\int_{\{0,1\}^{\mathbb N}} \, 2^{M_k(\omega,x)} \, dm(\omega)=\sum_{[\omega_1 \cdots \omega_k]} 2^{M_k(\omega,x)} 2^{-k}.\]
We now give a similar formula for the asymmetric case.

\begin{lem} \label{counting}
$\mathcal N_{k}(x,\beta)=\sum_{[\omega_1 \cdots \omega_k]} p^{(k-M_k(\omega,x))-\sum_{i=M_k(\omega,x)+1}^k \omega_i} (1-p)^{\sum_{i=M_k(\omega,x)+1}^k \omega_i}$.
\end{lem}

\begin{proof} We use a similar argument as the one used for the symmetric case (see \cite{Kem12}). Define
\[ \Omega(k,x)=\big\{\omega_1 \cdots \omega_{M_k(\omega,x)}: \omega\in \Omega \big\}.\]
If $x$ has a unique expansion, then $\Omega(k,x)$ consists of one element, the empty word. We now have $|\Omega(k,x)|= \mathcal N_k(x,\beta)$, and 

\begin{eqnarray*}
& &
 \sum_{[\omega_1 \cdots \omega_k]} p^{(k-M_k(\omega,x))-\sum_{i=M_k(\omega,x)+1}^k \omega_i} (1-p)^{\sum_{i=M_k(\omega,x)+1}^k \omega_i} \\
& = & \sum_{[\omega_1 \cdots \omega_k]} \frac{p^{k-\sum_{i=1}^k \omega_i}(1-p)^{\sum_{i=1}^k \omega_i}}{p^{M_k(\omega,x)-\sum_{i=1}^{M_k(\omega,x)} \omega_i}(1-p)^{\sum_{i=1}^{M_k(\omega,x)} \omega_i}}\\
& = & \int_{\Omega}\frac{1}{m_p([\omega_1 \cdots \omega_{M_k(\omega,x)}])} dm_p(\omega) =  \sum_{j=0}^k \int_{\{\omega: M_k(\omega,x)=j\}}\frac{1}{m_p([\omega_1 \cdots \omega_j])} dm_p(\omega)\\
& = & \sum_{j=0}^k\,\,\sum_{\omega_1 \cdots \omega_j \in \Omega(k,x)}\frac{1}{m_p([\omega_1 \cdots \omega_j])} m_p([\omega_1 \cdots \omega_j]) =  |\Omega(k,x)|=\mathcal N_k(x,\beta). \quad \quad \qedhere
\end{eqnarray*}
\end{proof}

\begin{thm}\label{bounds}
For $\lambda$-a.e.~$x \in [0, \beta]$ one has
\[ \frac{-\big(\log (\max(p,1-p))+\gamma \big)}{\log \beta}\le \underline{d}(\mu_{\beta,p},x)\le \overline{d}(\mu_{\beta,p},x)\le \frac{-\big(\log (\min(p,1-p))+\gamma \big)}{\log \beta},\]
where $\displaystyle\lim_{k \to\infty}\frac{\log \mathcal N_k(x,\beta)}{k}=\gamma$ is the constant from (\ref{q:gamma}).
\end{thm}

\begin{proof} We use the same metric $\bar \rho$ on $[0,\beta]$ as in the previous section. Then
\begin{multline*}
\mu_{\beta,p}\big(B_{\bar \rho}(x,\beta^{-k})\big) = \sum_{[\omega_1 \cdots \omega_k]}\nu_{\beta}(B_{\rho}\big([\omega_1 \cdots \omega_k],x),\beta^{-k})\big)\\ 
= \sum_{[\omega_1 \cdots \omega_k]}Q_p([\alpha_1(\omega,x) \cdots \alpha_k(\omega,x)]) p^{(k-M_k(\omega,x))-\sum_{i=M_k(\omega,x)+1}^k \omega_i} \\
\cdot(1-p)^{\sum_{i=M_k(\omega,x)+1}^k \omega_i} u_{e_{1-\omega_k}}.
\end{multline*}
Now,
\[ Q_p([\alpha_1(\omega,x) \cdots \alpha_k(\omega,x)])=u_{\alpha_1(\omega,x)} p^{L_k(\omega,x)}(1-p)^{k-L_k(\omega,x)},\]
where
\[ L_k(\omega,x)=\#\{1\le j\le k:\alpha_j(\omega,x)=e_0\}+\#\{1\le j\le k:\alpha_j(\omega,x)\alpha_{j+1}(\omega,x)=e_1s\},\]
and hence
\begin{multline*}
k-L_k(\omega,x)\\
=\#\{1\le j\le k:\alpha_j(\omega,x)=e_1\}+\#\{1\le j\le k:\alpha_j(\omega,x)\alpha_{j+1}(\omega,x)=e_0s\}.
\end{multline*}
Let $C_1=\max(u_{e_0},u_s,u_{e_1})$ and $C_2=\min(u_{e_0},u_s,u_{e_1})$. Then, from Lemma (\ref{counting}) we have
\[ C_2\big(\min(p,1-p))\big)^k \mathcal N_k(x,\beta)\le \mu_{\beta,p}(B_{\bar \rho}(x,\beta^{-k}))\le C_1\big(\max(p,1-p))\big)^k \mathcal N_k(x,\beta).\]
Since $\beta$ is a Pisot number, $\displaystyle\lim_{k \to\infty}\frac{\log \mathcal N_k(x,\beta)}{k}=\gamma$ exists $\lambda$-a.e.~(see \cite{FS11}) and we have
\[ \frac{-[\log (\max(p,1-p))+\gamma]}{\log \beta}\le \underline{d}(\mu_{\beta,p},x)\le  \overline{d}(\mu_{\beta,p},x)\le \frac{-[\log (\min(p,1-p))+\gamma]}{\log \beta}. \qedhere\]
\end{proof}

\begin{rem}{\rm (i) If $p=1/2$, then both sides of the inequality in Theorem (\ref{bounds}) are equal to $\displaystyle\frac{\log 2-\gamma}{\log \beta}$ leading to 
\[ d(\mu_{\beta,1/2},x)= \frac{\log 2-\gamma}{\log \beta}\]
a.e.~as we have seen earlier.\\
(ii) We now consider the extreme cases $x\in \{0,\beta\}$, the only two points with a unique expansion. We begin with $x=\beta$. In this case
\[ Q_p([\alpha_1(\omega,\beta) \cdots \alpha_k(\omega,\beta)])=Q_p([e_1 \cdots e_1])=u_{e_1}(1-p)^k,\]
and $\mathcal N_k(\beta,\beta)=1$, so that
\[ C_2(1-p)^k\le \mu_{\beta,p}(B_{\bar \rho}(\beta,\beta^{-k}))\le C_1(1-p)^k.\]
Hence,
\[ d(\mu_{\beta,p}, \beta)=\frac{-\log (1-p)}{\log\beta}\]
for all $\omega\in \Omega$. A similar argument shows that
\[ d(\mu_{\beta,p},0)=\frac{-\log (p)}{\log\beta}\]
for all $\omega\in \Omega$.
}\end{rem}

\bigskip
\footnotesize
\noindent\textit{Acknowledgments.}
The second author was supported by NWO (Veni grant no. 639.031.140).

\bibliographystyle{alpha}
\bibliography{bernoulli}

\end{document}